\documentclass[11 pt]{amsart}
\usepackage{amscd,amsfonts,amssymb,amsmath}
\usepackage{hyperref}
\usepackage[margin=3.3cm]{geometry}

\newtheorem{theorem}{Theorem}[section]

\newtheorem{lemma}[theorem]{Lemma}

\theoremstyle{definition}

\theoremstyle{plain}
\numberwithin{equation}{section}
\newtheorem*{ack}{Acknowledgement}

\newtheorem{corollary}[theorem]{Corollary}

\def \Z{\mathbb Z}

\def \r{\rho}
\def \s{\sigma}

\newcommand{\secref}[1]{Section~\ref{#1}}
\newcommand{\thmref}[1]{Theorem~\ref{#1}}
\newcommand{\lemref}[1]{Lemma~\ref{#1}}

\newcommand{\corref}[1]{Corollary~\ref{#1}}

\newcommand{\eqnref}[1]{~{\textrm(\ref{#1})}}
\numberwithin{equation}{subsection}

\begin{document}
\title[Commutator Subgroups of Welded Braid Groups]{Commutator Subgroups of Welded Braid Groups}
\author[Soumya Dey]{Soumya Dey}
\author[Krishnendu Gongopadhyay]{Krishnendu Gongopadhyay}
\address{Indian Institute of Science Education and Research (IISER) Mohali, Sector 81,  SAS Nagar, P. O. 
Manauli, Punjab 140306, India.}
\email{soumya.sxccal@gmail.com} 
\address{Indian Institute of Science Education and Research (IISER) Mohali, Sector 81,  SAS Nagar, P. O. 
Manauli, Punjab 140306, India.}
\email{krishnendu@iisermohali.ac.in, krishnendug@gmail.com} 
\thanks{The authors  acknowledge partial support from the DST projects INT/RUS/RSF/P-2 and  DST/INT/JSPS/P-192/2014. Dey was  supported by a UGC JRF during the course of this work.   }
\date{\today}

\begin{abstract}
Let $WB_n$ be the welded (or loop) braid group on n strands, $n \geq 3$. We investigate commutator subgroup of $WB_n$.   We prove that the commutator subgroup $WB_n'$ is finitely generated and Hopfian. We show that $WB_n'$ is perfect if and only if $n \geq 5$. We also compute finite presentation for $FWB_n'$, the commutator subgroup of the flat welded braid group $FWB_n$. Along the way, we  investigate adorability of these groups. 
\end{abstract}
\maketitle

\section{Introduction}
\emph{Welded braid groups} are certain extensions of the classical braid groups.  These groups have appeared in several contexts in the literature and often with different names,  e.g.  \cite{bcw}, \cite{frr}, \cite{lk1},   \cite{collins}. They are also known as \emph{loop braid groups} or \emph{permutation braid groups} or \emph{symmetric automorphisms of free groups}.  We refer to the recent survey article by Damiani \cite{dam} for further details on different definitions and applications of these groups.

\medskip  In this paper, we investigate the commutator subgroups of the welded braid groups.  The commutator $B_n'$ of the classical braid group $B_n$ is well studied. 
Gorin and Lin \cite{gl} obtained a finite presentation of  $B_n'$. Simpler presentation of $B_n'$ was obtained by Savushkina \cite{sa}. Several authors have investigated the larger class of spherical Artin
groups, e.g.  \cite{zinde}, \cite{mr},  \cite{orevkov}. In this context, it is a natural question to investigate structures of commutator subgroups of other classes of generalized braid groups. 

\medskip 
 
Let $WB_n$ denote the welded braid group of $n$ strands. We investigate the commutator subgroup of $WB_n$. Recall that a group $G$ is called \emph{perfect} if it is equal to its commutator subgroup. We prove the following:
\begin{theorem}\label{mainth} Let $WB_n'$ denote the commutator subgroup of the welded braid group $WB_n$.

	\begin{itemize}
		
		\item[(i)] $WB_n'$ is a finitely generated group for all $n \ge 3$. For $n \geq 7$, the rank of $WB_n'$ is at most $1+2(n-3)$, and for $3 \le n \leq 6$, the rank is at most $4+2(n-3)$.  

		\medskip \item[(ii)]  For $n \geq 5$, $WB_n'$ is perfect.
	\end{itemize}
\end{theorem}

We note here that, $WB_2=F_2 \rtimes S_2$. So, the commutator $WB_2'$ is infinitely generated. 

\medskip Recall that a group $G$ is called \emph{Hopfian} if every epimorphism $G \to G$ is an  isomorphism. In general, being Hopfian is not a subgroup-closed group property. Using the above theorem, we have the following. 
\begin{corollary}\label{cor1}
For any $n \geq 3$, $WB_n'$ is Hopfian. 
\end{corollary}
Another consequence of the above theorem is the following.
\begin{corollary}\label{cor2}
For a free group $F_k$, the image of any nontrivial homomorphism $\phi : WB_n \rightarrow F_k$ is infinite cyclic. \end{corollary}

\medskip The Reidemeister-Schreier method is a standard technique to obtain presentations of subgroups; for details see \cite{mks}. This method was used to obtain presentations of certain classes of Artin groups in  \cite{man}, \cite{lo}, \cite{mr}. We shall use this method to compute a presentation for $WB_n'$. We shall first find out a presentation using  Reidemeister-Schreier method, and then using Tietze transformations, will eliminate redundant generators to obtain a finite generating set. 
\subsubsection*{Adorability of welded braid groups} Motivated by the covering theory of aspherical $3$-manifolds, Roushon defined the notion of an adorable group:  a group $G$ is called \emph{adorable} if $G^i/G^{i+1}=1$ for some $i$, where $G^i=[G^{i-1}, G^{i-1}]$ and $G^0=G$ are the terms in the derived series of $G$. The smallest $i$ for which the above property holds, is called the \emph{degree of adorability} of $G$. For more details on adorable groups, see \cite{rou, rou2}. Applying Roushon's results with part (ii) of the above theorem, we have the following. 
\begin{corollary}  \label{cor3}
For $n \geq 5$,  $WB_n$ is adorable of degree 1, and for $n=3, 4$, $WB_n$ is not adorable.  
\end{corollary}
Therefore, by \thmref{mainth} and \corref{cor3} we immediately have the following.
\begin{corollary}
The group $WB_n'$ is perfect if and only if  $n \geq 5$. 
\end{corollary}
This  generalizes the fact that $B_n$ is adorable of degree 1 for $n \geq 5$. It is easy to see that if $f:G \to H$ be a surjective homomorphism with $G$ adorable, then $H$ is also adorable and $doa(H) \leq doa(G)$, where $doa(G)$ denotes degree of adorability, see  \cite[Lemma 1.1]{rou2}.  It follows from \cite[Proposition 8]{bb}  that the commutator subgroup  $VB_n'$ of the virtual braid group $VB_n$ is perfect for $n \geq 5$. Thus, for $n \geq 5$, $VB_n$ is adorable of degree 1. The welded braid groups being quotients of these groups, are also adorable with degree $\leq 1$. This gives a proof of the fact that $WB_n'$ is perfect for $n \geq 5$ even without using the presentation of $WB_n'$. Using the presentation, we give a direct proof of this fact.

\medskip After finishing this article, we have come to know about the recent work of Zaremsky \cite{mz} that implies the finite presentability of $WB_n'$ for $n \geq 4$, see \cite[Theorem B]{mz}. The finite generation of this group for $n \geq 3$ is also implicit in this work. However, Zaremsky has not obtained any bound on the rank, nor he has observed properties like perfectness, Hopfianness etc. Zaremsky has used Morse theory of complex symmetric graphs to obtain his results. It would be interesting to obtain an explicit finite presentation of $WB_n'$. 

\medskip As a by-product of the method that we have used for $WB_n'$,  in \secref{flat}, we compute explicit presentation of the commutator subgroup of a  particular quotient of the welded braid group, known as the flat welded braid group, introduced by Kauffman and studied in \cite{kala}, \cite{dam2}. 

\section{Proof of \thmref{mainth}} 

\subsection{Presentation of $WB_n$} Recall that the group $WB_n$ is generated by a set of $2(n-1)$ generators: $\{\sigma_i, \rho_i, \ i=1, 2, \ldots, n-1\}$ satisfying the following set of defining relations:

\begin{enumerate}
	\item {\it The braid relations}: $$\sigma_i \sigma_j=\sigma_j \sigma_i \hbox{ if } |i-j|>1;$$
	$$\sigma_i \sigma_{i+1} \sigma_i=\sigma_{i+1} \sigma_i \sigma_{i+1};$$ 
	
	\item {\it The  symmetric relations}: 
	$$\rho_i^2=1;$$ 
	$$\rho_i \rho_j=\rho_j \rho_i, \hbox{ if } |i-j|>1;$$
	$$\rho_i \rho_{i+1} \rho_i=\rho_{i+1} \rho_i \rho_{i+1};$$
	
	\item {\it The mixed relations}: $$\sigma_i \rho_j =\rho_j \sigma_i, \hbox{ if } |i-j|>1;$$
	$$\rho_i \rho_{i+1} \sigma_i =\sigma_{i+1} \rho_i \rho_{i+1};$$
	
	\item{\it The forbidden relations}: $$\r_i \s_{i+1} \s_i=\s_{i+1} \s_i \r_{i+1}.$$
\end{enumerate}

\bigskip If we add the \emph{flat} relation: $\sigma_i^2=1, \ 1 \leq i \leq n-1$,  in the above presentation, we get the \emph{flat welded braid group}, denoted by $FWB_n$. The \emph{flat virtual braid group}, denoted by $FVB_n$, is an extension of $FWB_n$, which has a presentation that is obtained by removing the forbidden relations from the presentation of $FWB_n$.

\subsection{Computing the generators:}

For $n \geq 3$, define the map $\phi$: 
\begin{equation*}\label{se1}1 \xrightarrow {} WB_n' \xrightarrow{} WB_n \xrightarrow{\phi} \Z \times \Z_2 \xrightarrow{} 1\end{equation*}
where, for $i=1, \ldots, n-1$, $\phi(\sigma_i)=\overline{\sigma_1}$ , $\phi(\rho_i)=\overline{\rho_1}$; here  $\overline{\sigma_1}$ and $\overline{\rho_1}$ are the generators of $\Z$ and $\Z_2$ respectively when viewing it in the abelianization of $WB_n$. We will denote $\phi(a)$, for $a \in WB_n$,  simply by $\overline a$.

Here, Image($\phi$) is isomorphic to the abelianization of $WB_n$, denoted as $WB_n^{ab}$. To prove this, we abelianize the above presentation of $WB_n$ by inserting the relations $ ~ xy=yx ~ $ in the presentation for all $x,y \in \{ ~ \sigma_i, \rho_i ~ | ~ 1 \le i \le n-1 \ \} $. The resulting presentation is the following:
$$ WB_n^{ab} = ~ < \sigma_1, \rho_1 ~ | ~ \sigma_1 \rho_1 = \rho_1 \sigma_1, ~ \rho_1^2=1 ~ >$$
Clearly, $WB_n^{ab}$ is isomorphic to $\Z \times \Z_2$. But as $\phi$ is onto, Image($\phi$) = $\Z \times \Z_2$.  
Hence, Image($\phi$) is isomorphic to $WB_n^{ab}$.  Hence, $\phi$ defines the above short exact sequence and  $\phi$ does have a section. 
\begin{lemma}
 $WB_n'$ is generated by $\alpha_{m, \epsilon, i}=\sigma_1^m \rho_1^{\epsilon} \sigma_i \rho_1^{\epsilon} \sigma_1^{-1} \sigma_1^{-m}$ and $\beta_{m, \epsilon, i}=\sigma_1^m \rho_1^{\epsilon} \rho_i \rho_1 \rho_1^{\epsilon} \sigma_1^{-m}$, 
where $m \in \Z$, $\epsilon \in \{0, 1\}$, $1 \le i \le n-1$.
\end{lemma}
\begin{proof} Consider a Schreier set of coset representatives:
$$\Lambda=\{ \sigma_1^m \rho_1^{\epsilon} \ | \ m \in \Z, \epsilon \in \{0, 1\} \}.$$
By \cite[Theorem 2.7]{mks}, the group $WB_n'$ is generated by the set
$$\{S_{\lambda, a}=(\lambda a) (\overline{\lambda a})^{-1} \ | \ \lambda \in \Lambda, \ a \in \{\sigma_i, \rho_i | \ i=1, 2, \ldots, n-1\} \}.$$
Choose $\lambda=\sigma_1^m \rho_1^{\epsilon}$ from $\Lambda$. \\
For  $a=\sigma_i$, $S_{\lambda, a} = \sigma_1^m \rho_1^{\epsilon} \sigma_i \rho_1^{\epsilon} \sigma_1^{-1} \sigma_1^{-m}$. For $a=\rho_i$, $S_{\lambda, a} = \sigma_1^m \rho_1^{\epsilon} \rho_i \rho_1 \rho_1^{\epsilon} \sigma_1^{-m}$. \\

Hence, $WB_n'$ is generated by the following elements:
\begin{equation*}\alpha_{m, \epsilon, i}=S_{\sigma_1^m \rho_1^{\epsilon}, \sigma_i}=\sigma_1^m \rho_1^{\epsilon} \sigma_i \rho_1^{\epsilon} \sigma_1^{-1} \sigma_1^{-m},\end{equation*}
\begin{equation*}\beta_{m, \epsilon, i}=S_{\sigma_1^m \rho_1^{\epsilon}, \rho_i}=\sigma_1^m \rho_1^{\epsilon} \rho_i \rho_1 \rho_1^{\epsilon} \sigma_1^{-m},\end{equation*}
where $m \in \Z$, $\epsilon \in \{0, 1\}$, $1 \le i \le n-1$.\end{proof}
 
\subsection{Computing the defining relations:} To obtain defining relations of $WB_n'$, we define a re-writing process $\tau$ as below. Refer \cite{mks} for further details.
$$\tau(a_{i_1}^{\epsilon_1} \dots a_{i_p}^{\epsilon_p}) = S_{K_{i_1},a_{i_1}}^{\epsilon_1} \dots S_{K_{i_p},a_{i_p}}^{\epsilon_p} \hbox{ with } \epsilon_j = 1 \hbox{ or } -1,$$
where if $\epsilon_j = 1$, $K_{i_1} = 1$ and $K_{i_j}$ = $\overline{a_{i_1}^{\epsilon_1} \dots a_{i_{j-1}}^{\epsilon_{j-1}}}, ~ j \ge 2$,  \\ and if $\epsilon_j = -1$, $K_{i_j}$ = $\overline{a_{i_1}^{\epsilon_1} \dots a_{i_j}^{\epsilon_j}}$ .

By \cite[Theorem 2.9]{mks}, the group $WB_n'$ is defined by the relations:
$$\tau_{\mu, \lambda}=\tau(\lambda r_{\mu} \lambda^{-1})=1, ~ \lambda \in \Lambda,$$
where $r_{\mu}$ are the defining relators of $WB_n$:

$$r_1 = \sigma_i \sigma_j \sigma_i^{-1} \sigma_j^{-1} = 1, \ |i-j|>1;$$
$$r_2 = \sigma_i \sigma_{i+1} \sigma_i \sigma_{i+1}^{-1} \sigma_i^{-1} \sigma_{i+1}^{-1}=1;$$
$$r_3 = \rho_i^2=1;$$
$$r_4 = \rho_i \rho_j \rho_i \rho_j = 1, \ |i-j|>1;$$
$$r_5 = \rho_i \rho_{i+1} \rho_i \rho_{i+1} \rho_i \rho_{i+1} = 1;$$
$$r_6 = \sigma_i \rho_j \sigma_i \rho_j^{-1} = 1, \ |i-j|>1;$$
$$r_7 = \rho_i \rho_{i+1} \sigma_i \rho_{i+1} \rho_i \sigma_{i+1}^{-1} = 1;$$
$$r_8 = \rho_i \sigma_{i+1} \sigma_i \rho_{i+1} \sigma_i^{-1} \sigma_{i+1}^{-1} = 1.$$

\begin{lemma}\label{lemma1}
	The generators  $\alpha_{k, \mu, r}, \beta_{k, \mu, r}$, $k \in \Z$, $\mu \in \{ 0,1 \} $, $1 \le r \le n-1$, of $WB_n'$ satisfy the following set of defining relations:
	
	\begin{equation}\label{1}  \alpha_{k, \mu, r} ~ \alpha_{k+1, \mu, s} ~ \alpha_{k+1, \mu, r}^{-1} ~ \alpha_{k, \mu, s}^{-1} =1, ~ |r-s| > 1; \end{equation}
	\begin{equation}\label{2} \alpha_{k, \mu, r} ~ \alpha_{k+1, \mu, r+1} ~ \alpha_{k+2, \mu, r}  = \alpha_{k, \mu, r+1} ~ \alpha_{k+1, \mu, r} ~ \alpha_{k+2, \mu, r+1};\end{equation}
	\begin{equation} \label{3}  \beta_{k, \mu, r} ~ \beta_{k, 1- \mu, r}  =1;\end{equation}  
	\begin{equation}\label{4}  (\beta_{k, \mu, r} ~ \beta_{k, \mu, s})^2=1,  ~ |r-s|>1, ~ r,s \ge 2;\end{equation}
    \begin{equation} \label{5} (\beta_{k, \mu, r} ~ \beta_{k, \mu, r+1})^3=1;\end{equation}
    \begin{equation} \label{6} \alpha_{k, \mu, r} ~ \beta_{k+1, 1- \mu, s} ~ \alpha_{k, 1- \mu, r}^{-1} ~ \beta_{k, \mu, s}=1, ~ |r-s|>1;\end{equation} 
    \begin{equation}\label{7} \alpha_{k, \mu, r} ~ \beta_{k+1, \mu, r+1} ~ \beta_{k+1, 1- \mu, r} ~ \alpha_{k, \mu, r+1}^{-1} ~ \beta_{k, \mu, r} ~ \beta_{k, 1- \mu, r+1}=1;\end{equation}
    \begin{equation} \label{8} \alpha_{k, \mu, r+1} ~ \alpha_{k+1, \mu, r} ~ \beta_{k+2, \mu, r+1} ~ \alpha_{k+1, 1- \mu, r}^{-1} ~ \alpha_{k, 1- \mu, r+1}^{-1} ~ \beta_{k, \mu, r}=1;\end{equation}
\begin{equation} \label{9} \alpha_{k, 0, 1} =1, ~ k \in \Z;\end{equation}
\begin{equation} \label{10} \alpha_{k, \mu, r} = \alpha_{0, 0, r} ~ , ~ k \in \Z, ~ \mu \in \{ 0,1 \} , ~ r \ge 3;\end{equation}
\begin{equation} \label{11} \beta_{k, \mu, 1} =1, ~ k \in \Z , ~ \mu \in \{ 0,1 \};\end{equation}
\begin{equation} \label{12} \beta_{k, 0, r} = \beta_{k, 1, r} ~ , ~ k \in \Z, ~ r \ge 3.\end{equation}
\end{lemma}
\medskip

\begin{proof}
Note that, $\eqnref{9},\eqnref{10},\eqnref{11},\eqnref{12}$ follow from the definitions of $\alpha_{k, \mu, r}$ and $\beta_{k, \mu, r}$.\\
By re-writing the conjugates of $r_i$ (by elements of $\Lambda$) we get the relations:
$\eqnref{1} - \eqnref{8}$.\\

Note that, $\tau(r_1) = \tau (\sigma_i \sigma_j \sigma_i^{-1} \sigma_j^{-1}) = S_{1, \sigma_i}S_{\sigma_1, \sigma_j}S_{\sigma_1, \sigma_i}^{-1}S_{1, \sigma_j}^{-1} = \alpha_{0,0,i} ~ \alpha_{1,0,j} ~ \alpha_{1,0,i}^{-1} ~ \alpha_{0,0,j}^{-1} $. \\
So, this gives the relation: $\alpha_{0,0,i} ~ \alpha_{1,0,j} ~ \alpha_{1,0,i}^{-1} ~ \alpha_{0,0,j}^{-1}=1, ~ |i-j|>1$.\\
Then we have, $\tau(\rho_1 r_1 \rho_1) = \tau (\rho_1 \sigma_i \sigma_j \sigma_i^{-1} \sigma_j^{-1} \rho_1)= S_{1, \rho_1}S_{\rho_1, \sigma_i}S_{\sigma_1 \rho_1, \sigma_j}S_{\sigma_1 \rho_1, \sigma_i}^{-1}S_{\rho_1, \sigma_j}^{-1}S_{\rho_1, \rho_1}$\\ 
$= \beta_{0,0,1} ~ \alpha_{0,1,i} ~ \alpha_{1,1,j} ~ \alpha_{1,1,i}^{-1} ~ \alpha_{0,1,j}^{-1} ~ \beta_{0,1,1}$.\\
This gives the relation: $\alpha_{0,1,i} ~ \alpha_{1,1,j} ~ \alpha_{1,1,i}^{-1} ~ \alpha_{0,1,j}^{-1}, ~ |i-j|>1 $ (using $\eqnref{11}$).\\
Similarly, $\tau(\sigma_1^k r_1 \sigma_1^{-k}) = \tau (\sigma_1^k \sigma_i \sigma_j \sigma_i^{-1} \sigma_j^{-1} \sigma_1^{-k}) = S_{\sigma_1^k, \sigma_i}S_{\sigma_1^{k+1}, \sigma_j}S_{\sigma_1^{k+1}, \sigma_i}^{-1}S_{\sigma_1^k, \sigma_j}^{-1}$\\
$= \alpha_{k,0,i} ~ \alpha_{k+1,0,j} ~ \alpha_{k+1,0,i}^{-1} ~ \alpha_{k,0,j}^{-1} $.\\
So, we have the relation: $\alpha_{k,0,i} ~ \alpha_{k+1,0,j} ~ \alpha_{k+1,0,i}^{-1} ~ \alpha_{k,0,j}^{-1}=1, ~ |i-j|>1.$\\
In a similar way, $\tau(\sigma_1^k \rho_1 r_1 \rho_1 \sigma_1^{-k}) = \tau (\sigma_1^k \rho_1 \sigma_i \sigma_j \sigma_i^{-1} \sigma_j^{-1} \rho_1 \sigma_1^{-k})$\\
$= S_{\sigma_1^k, \rho_1}S_{\sigma_1^k \rho_1, \sigma_i}S_{\sigma_1^{k+1} \rho_1, \sigma_j}S_{\sigma_1^{k+1} \rho_1, \sigma_i}^{-1}S_{\sigma_1^k \rho_1, \sigma_j}^{-1}S_{\sigma_1^k, \rho_1} = \alpha_{k,1,i} ~ \alpha_{k+1,1,j} ~ \alpha_{k+1,1,i}^{-1} ~ \alpha_{k,1,j}^{-1} $.\\
This gives the relation: $\alpha_{k,1,i} ~ \alpha_{k+1,1,j} ~ \alpha_{k+1,1,i}^{-1} ~ \alpha_{k,1,j}^{-1}=1, ~ |i-j|>1.$\\

Merging these 4 relations into one we get $\eqnref{1}$.\\

In a similar manner we re-write the conjugates of $r_2, r_3, \dots r_8$ by elements of $\Lambda$ and club them suitably to get the relations $\eqnref{2} - \eqnref{8}$.

\

So, we have a set of defining relations for $WB_n'$, namely relations $\eqnref{1} - \eqnref{12}$ in the generators $\alpha_{k, \mu, r} \ , ~ \beta_{k, \mu, r}$ for $k \in \Z$, $\mu \in \{ 0,1 \} $, $1 \le r \le n-1$. Hence,  \lemref{lemma1} is proved.
\end{proof}
Now, we will eliminate some of the generators and relations through Tietze transformations in order to get a finite set of generators for $WB_n'$.

\begin{lemma}\label{lem1}
For $n \ge 3$, the group $ ~ WB_n' ~ $ is generated by finitely many elements, namely $\alpha_{0,1,1}, ~ \alpha_{0,0,2}, ~ \alpha_{1,0,2},$ $ ~ \beta_{0,0,2}, ~ \alpha_{0,0,r}, ~ \beta_{0,0,r}$, $ 3 \le r \le n-1$. 
\end{lemma}
\begin{proof} From the relations $\eqnref{9},\eqnref{11}$ it is evident that the generators $\alpha_{k, 0, 1}$ and $\beta_{k, \mu, 1}$ are redundant and we can remove them from the set of generators.

Using $\eqnref{10}$ we can replace $\alpha_{k, \mu, r}$ by $\alpha_{0,0,r}$ for $r \ge 3$ and remove all $\alpha_{k, \mu, r}$ with either $k \ne 0$ or $\mu \ne 0$.

Using $\eqnref{12}$ we can remove $\beta_{k, 1, r}$ by replacing the same with $\beta_{k, 0, r}$ in all other relations.
	
From $\eqnref{7}$ we have $\alpha_{k,1,2} = \beta_{k,0,2} \alpha_{k,1,1} \beta_{k+1,1,2}$. We remove $\alpha_{k,1,2}$ by replacing this value in all other relations. After this replacement, from $\eqnref{8}$ we deduce:
$$\beta_{k+1,1,2}^{-1} ~ \alpha_{k,1,1}^{-1} ~ \beta_{k,0,2}^{-1} ~ \alpha_{k,0,2} ~ \beta_{k+2,0,2} = \alpha_{k+1,1,1}$$
From this relation, we can express $\alpha_{k,1,1}$ in terms of $\alpha_{0,1,1}, ~ \alpha_{k,0,2}, ~ \beta_{k,0,2}, ~ \beta_{k,1,2}$. We replace this value of $\alpha_{k,1,1}$ in all other relations and remove $\alpha_{k,1,1}$ for all $k \ne 0$.

For $\mu = 0, ~ r = 1$, \eqnref{7} becomes:
$$\beta_{k+1, 0, 2} ~ \alpha_{k, 0, 2}^{-1} ~ \beta_{k, 1, 2}=1 ~ \iff ~ \beta_{k+1, 0, 2} ~ \alpha_{k, 0, 2}^{-1} ~ \beta_{k, 0, 2}^{-1}=1 ~ ( \hbox{using} \eqnref{3}) $$
Using this we have $ \beta_{k,0,2} = \beta_{0,0,2} ~ \alpha_{0,0,2} ~ \alpha_{1,0,2} \dots \alpha_{k-1,0,2}$ for $k \ge 1$ and we have
$\beta_{k,0,2} = \beta_{0,0,2} ~ \alpha_{-1,0,2}^{-1} ~ \alpha_{-2,0,2}^{-1} \dots \alpha_{k,0,2}^{-1}$ for $k \le -1$. \\ As we have $ ~ \beta_{k,1,2} = \beta_{k,0,2}^{-1} ~ $, we can express $\beta_{k,0,2}$ and $\beta_{k,1,2}$ in terms of $\alpha_{k,0,2}, ~ \beta_{0,0,2}$ in all the other relations and remove all $\beta_{k,1,2}$ and all $\beta_{k,0,2}$ except $\beta_{0,0,2}$.

For $\mu = 0, ~ r=1 $, \eqnref{2} becomes:
$$\alpha_{k+1, 0, 2} = \alpha_{k, 0, 2} ~ \alpha_{k+2, 0, 2}$$
Using this we replace all $\alpha_{k,0,2}$ in terms of $\alpha_{0,0,2}, ~ \alpha_{1,0,2}$.

Lastly, if $n \ge 4$, using \eqnref{6} we can remove $\beta_{k,0,r}$ for $k \ne 0, ~ r \ge 3$.\\

Hence, we get a presentation of $WB_n'$ with $4 + 2(n-3)$ generators $ ~ \alpha_{0,1,1}, ~ \alpha_{0,0,2}, ~ \alpha_{1,0,2},$ $ ~ \beta_{0,0,2}, ~ \alpha_{0,0,r}, ~ \beta_{0,0,r}, ~ 3 \le r \le n-1$, and infinitely many defining relations. This proves finite generation of $WB_n'$ for all $n \ge 3$. 
\end{proof} 
Now, we treat the case $n \ge 7$. Notice that if $n \geq 7$, for every generator $\beta_{k, 0, r}$ with $r \ge 3$, there is at least one $\beta_{k, 0, s}$ with $s \ge 3$ and $|r-s|>1$. This helps us to improve the number of generators of $WB_n'$ for $n \geq 7$.

\begin{lemma}\label{Lemma2}
For $n \geq 7$, $WB_n'$ can be generated by $2(n-3)+1$ elements, namely $\beta_{0, 0, 2},~ \beta_{0,0,r}, ~ \alpha_{0,0,r}$ for ~ $3 \le r \le n-1$.
\end{lemma}
\begin{proof}
We proceed with an alternative elimination process here.

As before, we eliminate $\alpha_{k,0,1}, ~ \beta_{k, \mu, 1}, ~ \beta_{k,1,r}$ for all $k, \mu$, $ 3 \le r \le n-1$ and $\alpha_{k, \mu, r}$ with either $k \ne 0$ or $\mu \ne 0$, using the relations $\eqnref{9},\eqnref{10}$, $\eqnref{11},\eqnref{12}$.

Note that, $\alpha_{k,0,2} = \beta_{k,1,2} ~ \beta_{k+1,0,2}$ and $\alpha_{k,1,2} = \beta_{k,0,2} ~ \alpha_{k,1,1} ~ \beta_{k+1,1,2}$.
Also note that, $\alpha_{k,1,1} = \beta_{k,0,r} ~ \beta_{k+1,0,r}$.

At first, we replace $\alpha_{k,0,2}$ and $\alpha_{k,1,2}$ by $\beta_{k,1,2} ~ \beta_{k+1,0,2}$ and $\beta_{k,0,2} ~ \alpha_{k,1,1} ~ \beta_{k+1,1,2}$ in all the above relations and remove these generators from the set of generators.

Next, we replace $\alpha_{k,1,1}$ by $\beta_{k,0,r} ~ \beta_{k+1,0,r}$ (for every $3 \le r \le n-1$ ) in the current set of relations and remove these generators from the set of generators, and we have a new set of defining relations in the generators $\beta_{k,0,2}, ~ \beta_{k,1,2}, ~ \beta_{k,0,r}, ~ \alpha_{0,0,r}$, for all $k \in \Z $ and $ ~ 3 \le r \le n-1$.

Let us assume $k \ge 0$. The case of $k<0$ is similar. In the new set of relations, note that for $n \geq 7$, we have  $\beta_{k+1,0,r}=\alpha_{0,0,s}^{-1} ~ \beta_{k,0,r} ~ \alpha_{0,0,s}, ~ |r-s|>1, ~ r,s \ge 3$. \\
Hence, we have $\beta_{k,0,r}=\alpha_{0,0,s}^{-k} ~ \beta_{0,0,r} ~ \alpha_{0,0,s}^k$. Also, note that, $\beta_{0,1,2}=\beta_{0,0,2}^{-1}$. \\
We replace $\beta_{k,0,2}, ~ \beta_{k,1,2}, ~ \beta_{k,0,r}$ by $\alpha_{0,0,l}^{-k} ~ \beta_{0,0,2} ~ \alpha_{0,0,l}^k$, $\alpha_{0,0,l}^{-k} ~ \beta_{0,0,2}^{-1} ~ \alpha_{0,0,l}^k$, $\alpha_{0,0,s}^{-k} ~ \beta_{0,0,r} ~ \alpha_{0,0,s}^k$ in the current set of relations and remove $\beta_{0,1,2}, ~ \beta_{k,0,2}, ~ \beta_{k,1,2}, ~ \beta_{k,0,r}$, for all $k \ne 0$, from the set of generators.\\

This gives us a new set of defining relations in the $ ~ 2(n-3)+1 ~ $ generators $ ~ \beta_{0,0,2}, ~ \beta_{0,0,r},$ $ \alpha_{0,0,r}$ for $3 \le r \le n-1$. 
This proves the lemma. \end{proof}

\subsubsection*{Perfectness of the commutator subgroup:} 
\begin{lemma}\label{lem3} 
The group $WB_n'$ is perfect for $n \geq 5$. 
\end{lemma}
\begin{proof}
For $n \geq 5$, we abelianize the above presentation of $WB_n'$ by adding the extra relations $xy=yx$ for all $x$, $y$ in the generating set. \\ 
After abelianizing, putting $r=1, ~ s=3$ in $\eqnref{1}$ we get: $$\alpha_{k, \mu, 1} ~ \alpha_{k+1, \mu, 1}^{-1} ~ \alpha_{k+1, \mu, 3} ~ \alpha_{k, \mu, 3}^{-1} = 1$$
$$\iff \alpha_{k, \mu, 1} = \alpha_{k+1, \mu, 1} ~ \hbox{(using \eqnref{10})}.$$
Hence, we have $\alpha_{k, 1, 1} = \alpha_{0, 1, 1}, ~ k \in \Z$. Note that we already have $\alpha_{k, 0, 1} =1, ~ k \in \Z.$ \\
If we put $r=2, ~ s=4$ (here we use $n \ge 5$) in $\eqnref{1}$ we get:
$$\alpha_{k, \mu, 2} ~ \alpha_{k+1, \mu, 2}^{-1} ~ \alpha_{k+1, \mu, 4} ~ \alpha_{k, \mu, 4}^{-1} = 1$$
$$\iff \alpha_{k, \mu, 2} = \alpha_{k+1, \mu, 2} ~ \hbox{(using \eqnref{10})}.$$
This gives us: $\alpha_{k, 0, 2} = \alpha_{0, 0, 2}, ~ \alpha_{k, 1, 2} = \alpha_{0, 1, 2}, ~ k \in \Z$.\\
Putting $r=1, ~ \mu = 0$ in $\eqnref{2}$ we get:
$$\alpha_{k, 0, 1} ~ \alpha_{k+1, 0, 2} ~ \alpha_{k+2, 0, 1} = \alpha_{k, 0, 2} ~ \alpha_{k+1, 0, 1} ~ \alpha_{k+2, 0, 2} $$
$$\iff \alpha_{0, 0, 2} = \alpha_{0, 0, 2}^2 ~ ( \hbox{using} \eqnref{9} \hbox{ and } \alpha_{k, 0, 2} = \alpha_{0, 0, 2} ~ , ~ k \in \Z ).$$
So, we have $\alpha_{0, 0, 2} = 1 \implies \alpha_{k, 0, 2} = 1, ~ k \in \Z$.\\
For the case $r=1, ~ \mu = 1$ in $\eqnref{2}$ we have:
$$\alpha_{k, 1, 1} ~ \alpha_{k+1, 1, 2} ~ \alpha_{k+2, 1, 1} = \alpha_{k, 1, 2} ~ \alpha_{k+1, 1, 1} ~ \alpha_{k+2, 1, 2} $$
$$\iff \alpha_{0, 1, 1}^2 ~ \alpha_{0, 1, 2} = \alpha_{0, 1, 1} ~ \alpha_{0, 1, 2}^2 ~ ( \hbox{using } ~ \alpha_{k, 1, 1} = \alpha_{0, 1, 1} ~ , ~ \alpha_{k, 1, 2} = \alpha_{0, 1, 2} ~ , ~ k \in \Z ). $$
Hence, we have $\alpha_{0, 1, 1} = \alpha_{0, 1, 2}$.\\
If we put $r=2, ~ \mu = 1$ in $\eqnref{2}$ we get:
$$\alpha_{k, 1, 2} ~ \alpha_{k+1, 1, 3} ~ \alpha_{k+2, 1, 2} = \alpha_{k, 1, 3} ~ \alpha_{k+1, 1, 2} ~ \alpha_{k+2, 1, 3}$$
$$\iff \alpha_{0, 1, 2}^2 ~ \alpha_{0, 0, 3} = \alpha_{0, 1, 2} ~ \alpha_{0, 0, 3}^2 ~ ( \hbox{using } \eqnref{10} \hbox{ and } \alpha_{k, 1, 2} = \alpha_{0, 1, 2} ~ , ~ k \in \Z).$$
This implies: $\alpha_{0, 1, 2} = \alpha_{0, 0, 3}$.\\
For the case $r=2, ~ \mu = 0$ in $\eqnref{2}$ we have:
$$\alpha_{k, 0, 2} ~ \alpha_{k+1, 0, 3} ~ \alpha_{k+2, 0, 2} = \alpha_{k, 0, 3} ~ \alpha_{k+1, 0, 2} ~ \alpha_{k+2, 0, 3} $$
$$\iff \alpha_{0, 0, 3} = \alpha_{0, 0, 3}^2 ~ \hbox{(using \eqnref{10} and $ ~ \alpha_{k,0,2}=1, ~ \forall k \in \Z$)}.$$
So, we have $ ~ \alpha_{0, 1, 1} = \alpha_{0, 1, 2} = \alpha_{0, 0, 3} = 1$.\\
Putting $r \ge 3$ in $\eqnref{2}$ we get:
$$\alpha_{0, 0, r}^2 ~ \alpha_{0, 0, r+1} = \alpha_{0, 0, r} ~ \alpha_{0, 0, r+1}^2 ~ \hbox{(using \eqnref{10})}.$$
This implies $ ~ \alpha_{0, 0, r} = \alpha_{0, 0, r+1}, ~ r \ge 3$ and hence we have $\alpha_{k, \mu, r} = 1$ for all $k, \mu, r$.\\

Considering the case $r=2, ~ s \ge 4$ in $\eqnref{4}$ we get: $\beta_{k, \mu, 2}^2=1, ~ $ as $\beta_{k, \mu, s}^2=1, ~ s \ge 3$ (follows from $\eqnref{3}$ and $\eqnref{12}$).\\
Also note that, putting $r=1$ in $\eqnref{5}$ we have: $\beta_{k, \mu, 2}^3=1$.\\
Above two relations imply: $\beta_{k, \mu, 2}=1$.\\
Now if we put $r=2$ in $\eqnref{5}$, we get: $\beta_{k, \mu, 3}^3=1$. 
But then we have $\beta_{k, \mu, 3}^2=1$, which imply $\beta_{k, \mu, 3}=1$.\\
Similarly, using $\eqnref{5}$ iteratively we deduce: $\beta_{k, \mu, r} = 1 ~ $ for all $k, \mu, r$.\\

Hence, for $n \ge 5$, in the abelianization of $WB_n'$ the generators $\alpha_{k, \mu, r}$ and $\beta_{k, \mu, r}$ become identity and hence the abelianization of $WB_n'$ is trivial group.\\
This shows that for $n \geq 5$, $WB_n'$ is perfect. 
\end{proof}

\subsection{Proof of \thmref{mainth}  } \thmref{mainth} follows from \lemref{lem1}, \lemref{Lemma2} and \lemref{lem3}.

\subsection{Proof of the Corollaries}

\subsubsection{Proof of \corref{cor1}} Recall (see for instance \cite[Corollary 4.3]{dam}) that $WB_n$ is isomorphic to a subgroup of $Aut(F_n)$, and hence so also $WB_n'$. Since $F_n$ is residually finite, using a result of Magnus \cite{magnus} it follows that $Aut(F_n)$ is also residually finite.  Hence $WB_n'$ as a subgroup of $Aut(F_n)$ is also residually finite. It is well-known that a finitely generated residually finite group is Hopfian. Thus, using \thmref{mainth}, $WB_n'$ is Hopfian for all $n \ge 3$. 

\subsubsection{Proof of \corref{cor2}} 
	Suppose $\phi : WB_n \rightarrow F_k$ be a nontrivial homomorphism. By \thmref{mainth}, $WB_n'$ is finitely generated. Hence, $\phi (WB_n') = \phi (WB_n)'$ is finitely generated. But, $\phi (WB_n)$ is free group of finite rank. Hence, $\phi (WB_n)'$ is finitely generated only if rank of $\phi (WB_n)$ is at most 1. This proves \corref{cor2}.

\subsubsection{Proof of \corref{cor3}}  It follows from \cite{bardakov} that there is a non-trivial homomorphism from the pure welded braid group, $PWB_n$, onto a free group. Hence $PWB_n$ is not adorable. For $k=3, 4$, $WB_k/PWB_k$ is a finite solvable group. Hence, by \cite[Proposition 1.7]{rou2}, for $k=3, 4$, $WB_k$ is not adorable, in particular, $WB_k'$ is not perfect for $k=3, 4$.  This proves \corref{cor3}. 

\section{Commutators of Flat Braid Groups}\label{flat}
In this section, we prove the following theorem. 
\begin{theorem}\label{vfbn}
	$FVB_n'$ and $FWB_n'$ are finitely presented groups for all $n$.
\end{theorem}

Using arguments similar to the ones used for $WB_n'$, we have the following observation as well.  
\begin{corollary}\label{cor4}
The flat virtual braid group $FVB_n$ and the flat welded braid group $FWB_n$ are adorable groups of degree 1 for $n \geq 5$. In particular, commutator subgroups of these groups are perfect for $n \geq 5$. 
\end{corollary}

We shall prove \thmref{vfbn} in the rest of this section by deducing explicit finite presentations for $FVB_n'$ and $FWB_n'$ using Reidemeister-Schreier method and Tietze transformations. 

\subsection{Computing the generators:} Let $G=FVB_n$ or $FWB_n$. 
Define the map $\phi$: 
$$1 \xrightarrow {} G' \xrightarrow{} G \xrightarrow{\phi} \Z_2 \times \Z_2 \xrightarrow{} 1$$
where, for $i=1, \ldots, n-1$, $\phi(\sigma_i)=\overline{\sigma_1}$, $\phi(\rho_i)=\overline \rho_1$; here  $\overline \sigma_1$ and $\overline \rho_1$ are the generators of the two copies of $\Z_2$. We will denote $\phi(a)$, for $a \in G$, simply by $\overline a$. Note that, $\phi$ does have a section in the above short exact sequence and  $\ker \phi=G'$.\\

Here, Image($\phi$) is isomorphic to the abelianization of $G$, denoted as $G^{ab}$. To prove this, we abelianize the presentations of $FVB_n$ and $FWB_n$ by inserting the relations $ ~ xy=yx ~ $ in the presentations for all $x,y \in \{~ \sigma_i, \rho_i ~ | ~ 1 \le i \le n-1 ~ \} $.
We find that in both the cases the resulting presentation is the following:
$$ G^{ab} = ~ < \sigma_1, \rho_1 ~ | ~ \sigma_1 \rho_1 = \rho_1 \sigma_1, ~ \sigma_1^2=1, ~ \rho_1^2=1 ~ >.$$
Clearly, $G^{ab}$ is isomorphic to $\Z_2 \times \Z_2$ . But as $\phi$ is onto, Image($\phi$) =$ ~ \Z_2 \times \Z_2$.\\
Hence, Image($\phi$) is isomorphic to $G^{ab}$.

\begin{lemma}\label{lemma6}
	$G'$ is generated by $\sigma_i \sigma_1=a_i$, $\rho_i \rho_1 = b_i$, $\rho_1 \sigma_i \rho_1 \sigma_1 = c_i$, $\rho_1 \rho_i = d_i$, $\sigma_1 \sigma_i = e_i$, $\sigma_1 \rho_i \rho_1 \sigma_1 = f_i$, $\sigma_1 \rho_1 \sigma_i \rho_1 = g_i$, $\sigma_1 \rho_1 \rho_i \sigma_1 = h_i$ for $i=1, 2, \ldots, n-1$.
\end{lemma}

\begin{proof} Consider a Schreier set of coset representatives:
$$\Lambda=\{ 1, \sigma_1, \rho_1, \sigma_1 \rho_1 \}.$$
By \cite[Theorem 2.7]{mks}, the group $G'$ is generated by the set: 
$$\{S_{\lambda, a}=(\lambda a) (\overline{\lambda a})^{-1} \ | \ \lambda \in \Lambda, \ a \in \{\sigma_i, \rho_i \ | \ i=1, 2, \ldots, n-1\} \}.$$

We compute the generators as follows:
\begin{enumerate}
	\item {\it For $\lambda = 1$}: $$S_{1, \sigma_i}=\sigma_i \sigma_1 = a_i \ ,$$ $$S_{1, \rho_i}=\rho_i \rho_1 = b_i \ ;$$
	\item {\it For $\lambda = \rho_1$}: $$S_{\rho_1,\sigma_i}=\rho_1 \sigma_i \rho_1 \sigma_1 = c_i \ ,$$ $$S_{\rho_1,\rho_i}=\rho_1 \rho_i = d_i \ ;$$
	\item {\it For $\lambda = \sigma_1$}: $$S_{\sigma_1,\sigma_i}=\sigma_1 \sigma_i = e_i \ ,$$ $$S_{\sigma_1,\rho_i}=\sigma_1 \rho_i \rho_1 \sigma_1 = f_i \ ;$$
	\item {\it For $\lambda = \sigma_1 \rho_1$}: $$S_{\sigma_1 \rho_1,\sigma_i}=\sigma_1 \rho_1 \sigma_i \rho_1 = g_i \ ,$$ $$S_{\sigma_1 \rho_1,\rho_i}=\sigma_1 \rho_1 \rho_i \sigma_1 = h_i \ .$$
\end{enumerate}
This proves the lemma.
\end{proof}

\subsection{Computing the defining relations:} To obtain defining relations of $G'$, we define a re-writing process $\tau$ as before. By \cite[Theorem 2.9]{mks}, the group $G'$ is defined by the relations:
$$\tau_{\mu, \lambda}=\tau(\lambda r_{\mu} \lambda^{-1})=1, ~ \lambda \in \Lambda,$$
where $r_{\mu}$ are the defining relators of $G$.\\

\begin{lemma}\label{lemma7}
	$FVB_n'$ has the following finite presentation: \\ \\
	Set of Generators: $$c_1, c_2, f_2, a_i, b_i, \ i=2, \dots ,n-1$$
	Set of Defining Relations:
	$$a_2^3=b_2^3=c_2^3=f_2^3=1;$$
	$$a_i^2=b_i^2=(b_i c_1)^2=1, \ i=3,\dots,n-1;$$
	$$b_2^{-1} f_2 a_2^{-1}=1;$$
	$$b_2 c_1 f_2^{-1} c_2^{-1}=1;$$
	$$(a_2 a_i)^2=(b_2 b_i)^2=(c_2 a_i)^2=(f_2 b_i c_1)^2=1 , i \ge 4;$$
	$$(a_2 a_3)^3=(b_2 b_3)^3=(c_2 a_3)^3=(f_2 b_3 c_1)^3=1;$$
	$$a_2 b_i c_1=b_i c_2 , i \ge 4;$$
	$$a_i f_2=b_2 a_i , i \ge 4;$$
	$$b_2 b_3 a_2 b_3 c_1 f_2^{-1} a_3=1;$$
	$$b_2^{-1} b_3 c_2 b_3 c_1 f_2 a_3=1;$$
	$$(a_i a_j)^2 = (b_i b_j)^2 = 1, \ i,j \ge 3, \ |i-j|>1;$$
	$$(a_i a_{i+1})^3 = (b_i b_{i+1})^3 = 1, \ i \ge 3;$$
	$$b_j^{-1} a_i^{-1} b_j a_i = c_1, \ i,j \ge 3, \ |i-j|>1;$$
	$$b_i b_{i+1} a_i = a_{i+1} b_i b_{i+1}, \ i \ge 3.$$\\
	The group $FWB_n'$ is generated by the same set of generators as above and has a set of defining relations as the set of relations above along with the following relations: 
	
	$$a_2^{-1} c_2 c_1^{-1} b_2^{-1}=1;$$
	$$a_2 c_2^{-1} c_1 f_2^{-1}=1;$$
	$$a_2 a_3 b_2 a_3 c_2^{-1} b_3=1;$$
	$$a_2^{-1} a_3 f_2 a_3 c_2 b_3 c_1=1;$$
	$$a_2 b_i c_1=b_i c_2 , \  i \ge 4.$$
	
\end{lemma}

\begin{proof}
The defining relations for $FVB_n$ are:
$$r_1 = \sigma_i \sigma_j \sigma_i \sigma_j = 1, \ |i-j|>1,$$
$$r_2 = \sigma_i \sigma_{i+1} \sigma_i \sigma_{i+1} \sigma_i \sigma_{i+1}=1,$$
$$r_3 = \sigma_i^2=1,$$
$$r_4 = \rho_i^2=1,$$
$$r_5 = \rho_i \rho_j \rho_i \rho_j = 1, \ |i-j|>1,$$
$$r_6 = \rho_i \rho_{i+1} \rho_i \rho_{i+1} \rho_i \rho_{i+1} = 1,$$
$$r_7 = \sigma_i \rho_j \sigma_i \rho_j = 1, \ |i-j|>1,$$
$$r_8 = \rho_i \rho_{i+1} \sigma_i \rho_{i+1} \rho_i \sigma_{i+1} = 1.$$\\
There is one extra defining relation for $FWB_n$:
$$r_9 = \rho_i \sigma_{i+1} \sigma_i \rho_{i+1} \sigma_i \sigma_{i+1} = 1.$$\\
Now we rewrite the conjugates (by elements of $\Lambda$) of each of the defining relators of the above presentations of $FVB_n$ and $FWB_n$ in order to get the defining relations for the commutator subgroups, i.e. $FVB_n'$ and $FWB_n'$.\\

To start with, consider the first defining relation: $r_1 = \sigma_i \sigma_j \sigma_i \sigma_j = 1, \ |i-j|>1$.\\
We conjugate this relation by each element $\lambda \in \Lambda=\{ 1, \sigma_1, \rho_1, \sigma_1 \rho_1 \}$ and rewrite them as follows:

\begin{enumerate}
	\item { For $\lambda = 1$}: $1 =\tau(r_1)=\tau (\sigma_i \sigma_j \sigma_i \sigma_j)$ \\ 
	$ = S_{1, \sigma_i}S_{\sigma_1, \sigma_j}S_{1, \sigma_i}S_{\sigma_1, \sigma_j} = (a_i e_j)^2, ~ |i-j|>1;$\\
	
	\item { For $\lambda = \sigma_1$}: $1 =\tau(\sigma_1 r_1 \sigma_1)= \tau (\sigma_1 \sigma_i \sigma_j \sigma_i \sigma_j \sigma_1) $ \\ 
	$ = S_{1, \sigma_1}S_{\sigma_1, \sigma_i}S_{1, \sigma_j}S_{\sigma_1, \sigma_i}S_{1, \sigma_j} S_{\sigma_1, \sigma_1} = (e_i a_j)^2, ~ |i-j|>1;$\\
	
	\item { For $\lambda = \rho_1$}: $1 = \tau(\rho_1 r_1 \rho_1) = \tau (\rho_1 \sigma_i \sigma_j \sigma_i \sigma_j \rho_1)$ \\
	$ = S_{1, \rho_1}S_{\rho_1, \sigma_i}S_{\sigma_1 \rho_1, \sigma_j}S_{\rho_1, \sigma_i}S_{\sigma_1 \rho_1, \sigma_j} S_{\rho_1, \rho_1} = (c_i g_j)^2, ~ |i-j|>1;$\\
	\item { For $\lambda = \sigma_1 \rho_1$}: $1 = \tau(\sigma_1 \rho_1 r_1 \rho_1 \sigma_1) = \tau (\sigma_1 \rho_1 \sigma_i \sigma_j \sigma_i \sigma_j \rho_1 \sigma_1)$ \\
	$ = S_{1, \sigma_1}S_{\sigma_1, \rho_1}S_{\sigma_1 \rho_1, \sigma_i}S_{\rho_1, \sigma_j}S_{\sigma_1 \rho_1, \sigma_i}S_{\rho_1, \sigma_j} S_{\sigma_1 \rho_1, \rho_1} S_{\sigma_1, \sigma_1} = (g_i c_j)^2, ~ |i-j|>1.$\\
\end{enumerate}

In this way, we get some of the defining relations for $FVB_n'$ and $FWB_n' ~ $, namely \\ $(a_i e_j)^2=1, ~ |i-j|>1$ and $(c_i g_j)^2=1, ~ |i-j|>1$.\\

In a similar manner we rewrite the conjugates of other defining relators i.e. $r_2, r_3, \dots r_8$, and deduce the remaining defining relations for $FVB_n'$:\\
$$(a_i e_{i+1})^3=1,~ (e_i a_{i+1})^3=1,~ a_i e_i=1;$$
$$(b_i d_{i+1})^3=1, ~ (d_i b_{i+1})^3=1,~ b_i d_i=1;$$
$$(f_i h_{i+1})^3=1, ~ (h_i f_{i+1})^3=1,~ f_i h_i=1;$$
$$(c_i g_{i+1})^3=1, ~ (g_i c_{i+1})^3=1, ~ c_i g_i=1;$$
$$(b_i d_j)^2=1, ~ (f_i h_j)^2=1, ~ |i-j|>1;$$
$$a_i f_j g_i d_j=1, ~ e_i b_j c_i h_j=1, ~ |i-j|>1;$$
$$b_i d_{i+1} a_i f_{i+1} h_i e_{i+1}=1;$$
$$f_i h_{i+1} e_i b_{i+1} d_i a_{i+1}=1;$$
$$d_i b_{i+1} c_i h_{i+1} f_i g_{i+1}=1;$$
$$h_i f_{i+1} g_i d_{i+1} b_i c_{i+1}=1.$$\\
For $FWB_n'$ we have these extra relations (rewriting conjugates of $r_9$):
$$b_i c_{i+1} g_i d_{i+1} a_i e_{i+1}=1;$$
$$f_i g_{i+1} c_i h_{i+1} e_i a_{i+1}=1;$$
$$d_i a_{i+1} e_i b_{i+1} c_i g_{i+1}=1;$$
$$h_i e_{i+1} a_i f_{i+1} g_i c_{i+1}=1.$$\\
Now we use Tietze transformations to remove some of the generators.\\

Replacing $e_i,d_i,g_i,h_i$ by $a_i^{-1},b_i^{-1},c_i^{-1},f_i^{-1}$ respectively, we get the defining relations for $FVB_n'$ in terms of the generators $a_i,b_i,c_i,f_i$ as follows:\\
$$(a_i a_j^{-1})^2=1, \ |i-j|>1;$$
$$(b_i b_j^{-1})^2=1, \ |i-j|>1;$$
$$(c_i c_j^{-1})^2=1, \ |i-j|>1;$$
$$(f_i f_j^{-1})^2=1, \ |i-j|>1;$$\\
$$(a_i a_{i+1}^{-1})^3=1;$$
$$(b_i b_{i+1}^{-1})^3=1;$$
$$(c_i c_{i+1}^{-1})^3=1;$$
$$(f_i f_{i+1}^{-1})^3=1;$$\\
$$a_i f_j = b_j c_i, \ |i-j|>1;$$
$$a_1=b_1=f_1=1;$$\\
$$b_i b_{i+1}^{-1} a_i f_{i+1} f_i^{-1} a_{i+1}^{-1}=1;$$
$$b_i^{-1} b_{i+1} c_i f_{i+1}^{-1} f_i c_{i+1}^{-1}=1.$$\\
For $FWB_n'$ we have two extra defining relations:
$$a_i a_{i+1}^{-1} b_i c_{i+1} c_i^{-1} b_{i+1}^{-1}=1;$$
$$a_i^{-1} a_{i+1} f_i c_{i+1}^{-1} c_i f_{i+1}^{-1}=1.$$\\
Observe that, putting $i=1$ in the first 8 relations above, we get the following:
$$a_j^2=b_j^2=c_j^2=f_j^2=1, \ j=3,\dots,n-1,$$ $$a_2^3=b_2^3=c_2^3=f_2^3=1. $$\\
Note that, if $|i-j|>1$, we have $a_i f_j = b_j c_i$. Putting $j=1$, we get $c_i=a_i$ for $i=3,\dots,n-1$. And putting $i=1$, we get $f_j=b_j c_1$ for $j=3,\dots,n-1$.\\
We replace $c_i$ by $a_i$ and $f_i$ by $b_i c_1$ for $i=3,\dots,n-1$.\\

Putting $i=1$ in the next 2 relations, we have:
$$b_2^{-1} f_2 a_2^{-1}=1;$$
$$b_2 c_1 f_2^{-1} c_2^{-1}=1.$$\\
Putting $i=1$ in the two extra relations of $FWB_n'$, we have:
$$a_2^{-1} c_2 c_1^{-1} b_2^{-1}=1;$$
$$a_2 c_2^{-1} c_1 f_2^{-1}=1.$$\\
Similarly, considering the cases $i=2, j\ge4$, and, $i\ge4, j=2$, in the above relations we get the following set of relations for $FVB_n'$:
$$(a_2 a_i)^2=(b_2 b_i)^2=(c_2 a_i)^2=(f_2 b_i c_1)^2=1 , ~~i \ge 4;$$
$$(a_2 a_3)^3=(b_2 b_3)^3=(c_2 a_3)^3=(f_2 b_3 c_1)^3=1;$$
$$a_2 b_i c_1=b_i c_2 , ~~ i \ge 4;$$
$$a_i f_2=b_2 a_i ,~~ i \ge 4;$$
$$b_2 b_3 a_2 b_3 c_1 f_2^{-1} a_3=1;$$
$$b_2^{-1} b_3 c_2 b_3 c_1 f_2 a_3=1.$$\\
And, the two extra relations for $FWB_n'$:
$$a_2 a_3 b_2 a_3 c_2^{-1} b_3=1;$$
$$a_2^{-1} a_3 f_2 a_3 c_2 b_3 c_1=1.$$
Lastly, we consider the case $i,j \ge 3$:
$$(a_i a_j)^2 = (b_i b_j)^2 = 1, \ i,j \ge 3, \ |i-j|>1;$$
$$(a_i a_{i+1})^3 = (b_i b_{i+1})^3 = 1, \ i \ge 3;$$
$$b_j^{-1} a_i^{-1} b_j a_i = c_1, \ i,j \ge 3, \ |i-j|>1;$$
$$b_i b_{i+1} a_i = a_{i+1} b_i b_{i+1}, \ i \ge 3.$$\\
And, the two extra relations for $FWB_n'$:
$$a_i a_{i+1} b_i = b_{i+1} a_i a_{i+1}, \ i \ge 3;$$
$$c_1 a_i a_{i+1} b_i c_1 = b_{i+1} a_i a_{i+1}, \ i \ge 3.$$
This completes the proof of the lemma.
\end{proof}

Proof of \thmref{vfbn} follows from \lemref{lemma6} and \lemref{lemma7}. \\

In particular, for $n=3$, we have the following presentations: 
$$FVB_3'= \ < a,b,x,y \ |\ a^3=b^3=(ab)^3=(xy)^3=1,\ y^{-1}=axb>$$ 
$$FWB_3'= < a,b,c,x \ | \ a^3=b^3=c^3=1, abc=1, axc=bax=xcb >.$$

\begin{ack}
We gratefully acknowledge useful discussions with Valeriy Bardakov. We are thankful to Matt Zaremsky for his comments on this work and for letting us know about his work. We are thankful to Celeste Damiani for her suggestions after thoroughly going through this work. Finally, it is a pleasure to thank the referee for useful comments.
\end{ack}

\end{document}